\documentclass[11pt]{amsart}
\usepackage{verbatim}
\usepackage{rotating} 
\usepackage{amssymb}
\usepackage[bookmarks,colorlinks,breaklinks]{hyperref}  
\hypersetup{linkcolor=blue,citecolor=blue,filecolor=blue,urlcolor=blue} 
\usepackage{mathrsfs,amsfonts,amsmath,amssymb,epsfig,amscd,xy,amsthm} 
 
\newtheorem{theorem}{Theorem}[section]
\newtheorem{proposition}[theorem]{Proposition}

\newtheorem{lemma}[theorem]{Lemma}

\theoremstyle{plain}

\theoremstyle{remark}

\newtheorem{remark}[theorem]{Remark}

\newcommand{\C}{{\mathbb C}}

\newcommand{\Q}{{\mathbb Q}}

\newcommand{\N}{{\mathbb N}}

\newcommand{\Kbar}{\overline{K}}

\newcommand{\Qbar}{\bar{\Q}}

\DeclareMathOperator{\Gal}{Gal}

\newcommand{\bP}{{\mathbb P}}

\newcommand{\bG}{{\mathbb G}}

\newcommand{\bC}{{\mathbb C}}

\newcommand{\bA}{{\mathbb A}}

\newcommand{\lra}{\longrightarrow}

\newcommand{\cL}{\mathcal{L}}
\newcommand\cLbar {\mathcal{\overline{L}}}

\newcommand\OK {\Omega_K}
\renewcommand\O{{\mathcal O}}
\renewcommand\P{{\mathbb P}}
\newcommand{\hhat}{{\widehat h}}

\author{D.~Ghioca}
\address{
Dragos Ghioca\\
Department of Mathematics\\
University of British Columbia\\
Vancouver, BC V6T 1Z2\\
Canada
}
\email{dghioca@math.ubc.ca}

\author{K.~D.~Nguyen}
\address{
Khoa D.~Nguyen \\
Department of Mathematics\\
University of British Columbia\\
And Pacific Institute for The Mathematical Sciences\\ 
Vancouver, BC V6T 1Z2, Canada}
\email{dknguyen@math.ubc.ca}

\author{H.~Ye}
\address{
Hexi Ye\\
Department of Mathematics\\
University of British Columbia\\
Vancouver, BC V6T 1Z2\\
Canada
}
\email{yehexi@math.ubc.ca}

\keywords{Dynamical Manin-Mumford Conjecture, equidistribution of points of small height, symmetries of the Julia set of a rational function}
\subjclass[2010]{Primary: 37P05. Secondary: 37P30}

\begin{document}
	\title[Dynamical Manin-Mumford Conjecture]{The Dynamical Manin-Mumford Conjecture and the Dynamical Bogomolov Conjecture for split rational maps}
	
	
	\begin{abstract}
	We prove the Dynamical Bogomolov Conjecture for endomorphisms  $\Phi:\bP^1\times \bP^1\lra \bP^1\times \bP^1$, where $\Phi(x,y):=(f(x), g(y))$ for any rational functions $f$ and $g$ defined over $\Qbar$. We use the equidistribution theorem for points of small height with respect to an algebraic dynamical system, combined with a theorem of Levin regarding symmetries of the Julia set. Using a specialization theorem of Yuan and Zhang, we can prove the Dynamical Manin-Mumford Conjecture for endomorhisms $\Phi=(f,g)$ of $\bP^1\times \bP^1$, where $f$ and $g$ are rational functions defined over an arbitrary field of characteristic $0$. 
\end{abstract} 
	  
	\maketitle


\section{Introduction}\label{sec:introduction}

We prove the Dynamical Manin-Mumford Conjecture (over $\C$) and the Dynamical Bogomolov Conjecture (over $\Qbar$) for endomorphisms $\Phi$ of $\bP^1\times \bP^1$ (see \cite[Conjectures~1.2.1~and~4.1.7]{ZhangLec}). Actually, we can prove an even stronger result than the one conjectured in \cite{ZhangLec} since we do not assume $\Phi:=(f_1,f_2):\bP^1\times \bP^1\lra \bP^1\times \bP^1$ is necessarily polarizable (i.e., $f_1$ and $f_2$ might have different degrees), but we exclude the case when the $f_i$'s are conjugate to monomials, $\pm$Chebyshev polynomials, or Latt\`es maps since in those cases there are counterexamples to a formulation which does not ask that $\deg(f_1)=\deg(f_2)$ (see Remark~\ref{discussion DMM}). In Theorem~\ref{general DMM result} we prove the appropriately modified statement of the Dynamical Manin-Mumford Conjecture for all polarizable endomorphisms of $\bP^1\times \bP^1$.

First, we recall that the Chebyshev polynomial of degree $d$ is the unique polynomial $T_d$ with the property that for each $z\in\C$, we have $T_d(z+1/z) = z^d + 1/z^d$. Similarly, a Latt\`es map $f:\bP^1\lra \bP^1$ is a rational function coming from the quotient of an affine map $L(z)=az+b$ on a torus $\mathcal{T}$ (elliptic curve), i.e. $f=\Theta\circ L\circ \Theta^{-1}$ with $\Theta: \mathcal{T}\to \P^1$ a finite-to-one holomorphic map; see \cite{lattes} by Milnor. For two rational functions $f$ and $g$, we say they are  (linearly) conjugate if there exists an automorphism $\eta$ of $\bP^1$ such that $f = \eta^{-1}\circ g\circ \eta$; we note that, by definition, the class of Latt\`es maps is invariant under conjugation. Any rational map of degree $d>1$ which is conjugate either to $z^{\pm d}$, or to $\pm T_d(z)$, or to a Latt\`es map is called \emph{exceptional}. Finally, we note that two polynomials $f$ and $g$ are linearly conjugate if there exists a linear polynomial $\eta$ such that $f=\eta^{-1}\circ g\circ \eta$.

\begin{theorem}
\label{main result}
Let $f_1$ and $f_2$ be rational functions and let $C\subset \bP^1\times \bP^1$ be an irreducible curve defined over $\C$ which projects dominantly onto both coordinates. Assume $d_i:=\deg(f_i) >1$ for $i=1,2$, and also that for some $i=1,2$ we have that $f_i(z)$ is not exceptional.  Then $C$ contains infinitely many points $(x_1,x_2)$ such that $x_i$ is a preperiodic point for $f_i$ for $i=1,2$ if and only if there exist positive integers $\ell_1$ and $\ell_2$ such that the following two conditions are met:
\begin{enumerate}
\item[(i)] $d_1^{\ell_1}=d_2^{\ell_2}$; and 
\item[(ii)] $C$ is preperiodic under the action of $(x,y)\mapsto (f_1^{\ell_1}(x), f_2^{\ell_2}(y))$ on $\bP^1\times \bP^1$.
\end{enumerate}
Moreover, if $f_1$, $f_2$ and $C$ are defined over $\Qbar$, then there exist positive integers $\ell_1$ and $\ell_2$ such that conditions~(i)-(ii) are met if and only if there exists an infinite sequence of points $(x_{1,n},x_{2,n})\in C(\Qbar)$ such that $\lim_{n\to\infty}\hhat_{f_1}(x_{1,n})+\hhat_{f_2}(x_{2,n})=0$, where $\hhat_{f_i}$ is the canonical height with respect to the rational function $f_i$. 
\end{theorem}

Motivated by the classical Manin-Mumford conjecture (proved by Raynaud \cite{Raynaud} in the case of abelian varieties and by McQuillan \cite{McQuillan} in the general case of semiabelian varieties) and also by the classical Bogomolov conjecture (proved by Ullmo \cite{Ullmo} in the case of curves embedded in their Jacobian and by Zhang \cite{Zhang:Bogomolov} in the general case of abelian varieties), Zhang formulated dynamical analogues of both conjectures (see \cite[Conjecture 1.2.1, Conjecture 4.1.7]{ZhangLec}) for polarizable endomorphisms of any projective variety. We say that an endomorphism $\Phi$ of a projective variety $X$ is \emph{polarizable} if there exists an ample line bundle $\cL$ on $X$ such that $\Phi^*\cL$ is linearly equivalent to $\cL^{\otimes d}$ for some integer $d>1$. As initially conjectured by Zhang, one might expect that if $X$ is defined over a field $K$ of characteristic $0$ and $\Phi$ is a polarizable endomorphism of $X$, then a  subvariety $V\subseteq X$ contains a Zariski dense set of preperiodic points if and only if $V$ is preperiodic. We say that an irreducible subvariety $Y\subseteq  X$ is \emph{preperiodic} if $\Phi^m(Y)=\Phi^n(Y)$ for some integers $n>m\ge 0$; if $m=0$, then we say that $Y$ is \emph{periodic}. Furthermore if $K$ is a number field then one can construct the canonical height $\hhat_\Phi$  for all points in $X(\Qbar)$ with respect to the action of $\Phi$ (see \cite{C-S}) and then Zhang's dynamical version of the Bogomolov Conjecture asks that if a subvariety $V\subseteq X$ is not preperiodic, then there exists $\epsilon>0$ with the property that the set of points $x\in V(\Qbar)$  such that $\hhat_\Phi(x)<\epsilon$ is not Zariski dense in $V$. Since all preperiodic points have canonical height equal to $0$, the Dynamical Bogomolov Conjecture is a generalization of the Dynamical Manin-Mumford Conjecture when the algebraic dynamical system $(X,\Phi)$ is defined over a number field.

Besides the case of abelian varieties $X$ endowed with the multiplication-by-$2$ map $\Phi$ (which motivated Zhang's conjectures), there are known only a handful of special cases of the Dynamical Manin-Mumford or the Dynamical Bogomolov conjectures. All of these partial results are for endomorphisms of $\bP^1\times \bP^1$---see \cite{Baker-Hsia, GT-Bogomolov, GTZ} (also see \cite{Favre-Dujardin} for a proof of a version of the Dynamical Manin-Mumford Conjecture in the context of polynomial automorphisms of $\bA^2$). With the exception of \cite{Favre-Dujardin}, all of the above cited papers deal only with the case of linear subvarieties $V$ of $\bP^1\times \bP^1$. Our Theorem~\ref{main result} settles the Dynamical Manin-Mumford and the Dynamical Bogomolov conjectures for all subvarieties of $\bP^1\times \bP^1$. 

\begin{remark}
\label{discussion DMM}
A direct computation shows that any dominant endomorphism of $\P^1\times\P^1$ is always of the form  $(x_1,x_2)\mapsto (f_1(x_1), f_2(x_2))$ (up to a switch of $f_1(x_1)$ and $f_2(x_2)$) where $f_i$ are rational functions for $i=1,2$. On the other hand, for any dominant endomorphism $\Phi$ of $\bP^1\times \bP^1$, we have that $\Phi^2$ is indeed of the form $(x_1,x_2)\mapsto (f_1(x_1), f_2(x_2))$. Since  $\Phi$ and $\Phi^2$ share the same set of preperiodic points (and also the same set of points of small height since they induce the same dynamical system on $\bP^1\times \bP^1$), then we can restrict our attention only to endomorphisms of the form $(x_1,x_2)\mapsto (f_1(x_1), f_2(x_2))$. Furthermore, such an endomorphism is polarizable if and only if $f_1$ and $f_2$ have the same degree $d>1$. 

In the case when the maps $f_i$ are both  conjugate to monomials or $\pm$Chebyshev polynomials, then there are counterexamples to a Dynamical Manin-Mumford statement unless $\deg(f_1)=\deg(f_2)$. In the case the maps $f_i$ are both  Latt\`es maps, there are counterexamples to the original formulation of the Dynamical Manin-Mumford Conjecture even when $\deg(f_1)=\deg(f_2)$---see \cite{GTZ} and the subsequent reformulation of the Dynamical Manin-Mumford Conjecture from \cite[Conjecture~2.4]{GTZ} which takes into account those counterexamples. Intuitively, the problem with any exceptional rational map $f$ lies with the fact that there are \emph{more than usual} maps $g$ commuting with $f$; this occurs since the dynamical system $(\bP^1, f)$ is covered in the three cases of an exceptional rational map by another dynamical system $(G, \Phi)$, where $G$ is an algebraic group ($\bG_m$ in the case $f$ is a monomial or a $\pm$Chebyshev polynomial, and an elliptic curve when $f$ is a Latt\`es map), while $\Phi$ is an endomorphism of $G$.  However, when the maps $f_i$ are not exceptional, then one expected that the condition on the degrees of the maps $f_i$ might be weakened---this is exactly what we prove in our Theorem~\ref{main result}.
\end{remark}

Using Theorem~\ref{main result}, we prove \cite[Conjecture~2.4]{GTZ} for all endomorphisms of $\bP^1\times \bP^1$.

\begin{theorem}
\label{general DMM result}
Let $f_1$ and $f_2$ be rational functions of degree $d>1$ defined over $\C$, let $\Phi:\bP^1\times \bP^1\lra \bP^1\times \bP^1$ be defined by $\Phi(x_1,x_2)=(f_1(x_1), f_2(x_2))$, and let $C\subset \bP^1\times \bP^1$ be an irreducible curve. Then $C$ is preperiodic under the action of $\Phi$ if and only if there exist infinitely many smooth points $P=(x_1,x_2)\in C$ which are preperiodic under $\Phi$ and moreover such that the tangent space of $C$ at $P$ is preperiodic under the induced action of $\Phi$ on  ${\rm Gr}_1\left(T_{\bP^1\times \bP^1,P}\right)$, where $T_{\bP^1\times \bP^1, P}$ is the tangent space of $\bP^1\times \bP^1$ at $P$.
\end{theorem}

During the final stages of preparing our paper, we learned that independently (using a slightly different approach), Romain Dujardin, Charles Favre and William Gignac proved Theorem~\ref{general DMM result} in the special case when  each $f_i$ is a polynomial.

We also note that using a simple induction argument, one can immediately extend both Theorems~\ref{main result} and \ref{general DMM result} to the case of curves contained in $(\bP^1)^n$ endowed with the coordinatewise action of one-variable rational functions $f_1,\dots, f_n$. The idea is that for any curve $C\subset (\bP^1)^n$, if we let $C_i\subset (\bP^1)^{n-1}$ (for $i=1,\dots, n$) be the image of $C$ under the projection map $\pi_i:(\bP^1)^n\lra (\bP^1)^{n-1}$ on the  $(n-1)$ coordinates of $(\bP^1)^n$ with the exception of its $i$-th coordinate, then $C$ is an irreducible component of the intersection $\cap_i \pi_i^{-1}(C_i)$. The inductive hypothesis yields that each $C_i$ and thus  $\pi^{-1}(C_i)$ must be preperiodic under the action of the $f_i$'s and therefore $C$ is preperiodic.

The idea of our proof is as follows. Using the equidistribution theorem for points of small height on a variety (see \cite{C-L} for the case of curves and \cite{Yuan} for the case of a higher dimensional variety, and also \cite{BR} and \cite{favre-rivera06} for the case of $\bP^1$), we prove that if $C\subset \bP^1\times \bP^1$ contains infinitely many preperiodic points under the action of $(f_1, f_2)$ (or points of small canonical height in the case the dynamical system is defined over $\Qbar$), then the measures $\tilde{\mu}_i$ induced on $C$ by the invariant measures corresponding to the dynamical systems $(\bP^1, f_i)$ are equal. Using a careful study of the local analytic maps which preserve (locally) the Julia set of a rational map which is not Latt\`es, monomial, or Chebyshev polynomial, we obtain the conclusion of Theorem~\ref{main result}. Even though all of the previously known results on the Dynamical Manin-Mumford Conjecture (see \cite{Baker-Hsia, GT-Bogomolov, GTZ}) used the powerful equidistribution theorem, since those results were valid only when $C$ is a line, one always reduced the question to two rational functions which share the same Julia set. The classification of such pair of  rational functions was simple (see \cite{Levin}). In our case, there is a much weaker consequence of the equidistribution that we obtain: for each point $(x_1, x_2)$ on the curve $C$, we have that $x_1$ is preperiodic under the action of $f_1$ if and only if $x_2$ is preperiodic under the action of $f_2$. This consequence was known for quite some time (see \cite{Mimar} which publishes the findings of Mimar's PhD thesis \cite{Mimar-thesis} from almost 20 years ago). The novelty of our approach is the exploit of the local symmetries of the Julia set which allows us to settle completely the Dynamical Manin-Mumford and the Dynamical Bogomolov Conjectures for all endomorphisms of $\bP^1\times \bP^1$.

Theorem~\ref{main result} yields that if a curve $C$ contains infinitely many  points with both coordinates preperiodic under the action of $f_1$, respectively $f_2$, then $C$ must be preperiodic. In turn this yields that there are \emph{transversal} (i.e., not horizontal nor vertical) curves in $\bP^1\times \bP^1$ which are fixed by the action of $(f_1^{m_1},f_2^{m_2})$. Hence, the dynamics of $f_1$ and $f_2$ are \emph{non-orthogonal} (see \cite{Medvedev-Scanlon} for the definition of orthogonality in this context). In other words, $f_1$ and $f_2$ are related dynamically; our next result shows that in the family of unicritical polynomials which are not conjugate to a monomial or to a Chebyshev polynomial, each map is non-orthogonal only to maps which are conjugate to itself. 

\begin{theorem}
\label{theorem example}
Let $d_1, d_2\geq 2$ be integers, let $c_1,c_2\in\C$, let $f_i(z)=z^{d_i}+c_i$ be polynomials such that $c_i\ne 0$ and also $c_i\ne -2$ if $d_i=2$, and let $\Phi$ be the endomorphism of $\bP^1\times \bP^1$ induced by the coordinatewise action of $f_1$ and $f_2$. Then there exists a curve $C\subset \bP^1\times \bP^1$ which projects dominantly onto both coordinates and also such that $C$ contains infinitely many preperiodic points under the action of $\Phi$ if and only if $d_1=d_2$ and there exists a $(d_1-1)$-st root of unity $\zeta$ such that $c_2=\zeta\cdot c_1$.
\end{theorem}

We observe that since polynomials $f$ and $g$ of degree $d\ge 2$ in \emph{normal form} (i.e. they are monic and the coefficient of $x^{d-1}$ in both $f$ and $g$ equals $0$) are conjugate if and only if there exists a $(d-1)$-st root of unity $\zeta$ such that $g(x)=\zeta^{-1}f(\zeta x)$, then a polynomial of the form $x^d + c$ is conjugate to $x^d$ if and only if $c=0$, and it is conjugate to $\pm T_d(x)$ if and only if $d=2$ and $c=-2$. This is an immediate consequence of the fact that the coefficient of $x^{d-2}$ in $T_d(x)$ is always nonzero. 

The proof of Theorem~\ref{theorem example} employs the results of \cite{Medvedev-Scanlon} along additional techniques stemming from the study of  polynomial decomposition (see also \cite{Nguyen}). Essentially, now due to our Theorem~\ref{main result} and the classification of invariant curves under polynomial actions done by \cite{Medvedev-Scanlon}, one can relatively easy determine for any two polynomials $f_1$ and $f_2$ whether there exists a transversal curve $C\subset \bP^1\times \bP^1$ containing infinitely many points $(x_1,x_2)$ with each $x_i$ preperiodic under the action of $f_i$. Our Theorem~\ref{theorem example} is an example in this direction that we considered to be interesting by itself and also for its proof, since the family of unicritical polynomials is one of the most studied families of polynomials from the point of view of complex dynamics; e.g. see \cite{AKLS}. More importantly, lots of the dynamical phenomenons of the family of unicritical polynomials are inherited by general families. We state next another result, which is a consequence of our Theorem~\ref{general DMM result} coupled with \cite[Theorem~6.24]{Medvedev-Scanlon}.

\begin{theorem}
\label{same poly acting}
Let $f\in \C[x]$ be a polynomial of degree $d\ge 2$ which is not conjugate to $x^d$ or to $\pm T_d(x)$, and let $\Phi:\bP^1\times \bP^1\lra \bP^1\times \bP^1$ be defined by $\Phi(x,y)=(f(x), f(y))$. Let $C\subset \bP^1\times \bP^1$ be a curve defined over $\C$ which projects dominantly onto both coordinates. Then $C$ contains infinitely many preperiodic points under the action of $\Phi$ if and only if $C$ is an irreducible component of the zero locus of an equation of the form $\tilde{f}^n(x) = L(\tilde{f}^m(y))$, where $m,n\ge 0$ are integers, $L$ is a linear polynomial commuting with an iterate of $f$, and
$\tilde{f}$ is a non-linear polynomial of minimal degree  
commuting with an iterate of $f$.
\end{theorem}

The plan of our paper is as follows. In Sections~\ref{section Julia set}~and~\ref{section equidistribution setup} we setup our notation, state basic properties for the Julia set of a rational function, construct the heights associated to an algebraic dynamical system and define adelic metrized line bundles which are employed in the main equidistribution result (Theorem~\ref{Yuan equidistribution}), which we will then use in our proof. In Section~\ref{section equal measures} we prove that under the hypothesis of Theorem~\ref{main result}, the measures induced on the curve $C$ from the two dynamical systems $(\bP^1, f_i)$ (for $i=1,2$) are equal. In Section~\ref{section from identical to preperiodic} we show how to use the equality of these two measures to infer the preperiodicity of the curve. In Section~\ref{section proof of main results} we finalize the proofs for Theorems~\ref{main result}~and~\ref{general DMM result}. We conclude our paper with Section~\ref{section examples} in which we prove Theorems~\ref{theorem example}~and~\ref{same poly acting}.
  
\medskip  

  {\bf Acknowledgments.} We thank Laura DeMarco, Charles Favre, Holly Krieger, Matt Satriano, Joseph Silverman, Thomas Tucker, Junyi Xie  and Shouwu Zhang  
	for useful discussions. We are grateful to both referees for their many useful comments and suggestions which improved our presentation. The first and third authors are partially supported by NSERC and the second author is partially supported
	by a UBC and PIMS postdoctoral fellowship. 	


\section{Dynamics of a rational function}
\label{section Julia set}

Let $f:\bP^1\to \bP^1$ be a rational function defined over $\C$ of degree $d\geq 2$. As always in dynamics, we denote by $f^n$ the $n$-th compositional iterate of $f$.


\subsection{The Julia set}   A family of maps on $U$ is {\em normal} if for any sequence of maps in this family, we can always pick a subsequence converging locally uniformly on $U$.  The {\em Julia set} $J_f$ is the set of points on $\bP^1_\C$, where $\{f^n\}_{n\geq 1}$ is not  normal restricted on any of their neighbourhoods. The Julia set $J_f$  is closed, nonempty and invariant under $f$. Let $x$ be a periodic point in a cycle of exact period $n$; then the {\em multiplier} $\lambda$ of this cycle (or of the periodic point $x$) is the derivative of $f^n$ at $x$. A cycle is {\em repelling} if its multiplier has absolute value greater than $1$. All but finitely many cycles of $f$ are repelling, and repelling cycles are in the Julia set $J_f$. Locally, at a repelling fixed point $x$ with multiplier $\lambda$, we can conjugate $f$ to the linear map $z\to \lambda \cdot z$ near $z=0$ (note that $\lambda\ne 0$ since the point is assumed to be repelling). For more details about the dynamics of a rational function, we refer the reader to Milnor's book \cite{Milnor:book}. 


\subsection{The invariant measure} 
\label{subsection measures}

There is a probability measure $\mu_f$ on $\bP^1$ associated to $f$, which is the unique $f$-invariant measure achieving maximal entropy $\log d$; see \cite{Bro, Lyu1, FLM, Man}.  Also $\mu_f$ is the unique measure satisfying 
\begin{equation}\label{measure growth}
\mu_f(f(A))=d\cdot \mu_f(A)
\end{equation} for any Borel set $A\subset \bP^1_\C$ with $f$ injective restricted on $A$.  The support of $\mu_f$ is $J_f$, and $\mu_f(x)=0$ for any $x\in \bP^1_\C$. 


\subsection{Measures on a curve associated to a dynamical system}
\label{subsection measures 2}

Let $C\subset\bP^1_\C\times \bP^1_\C$ be an irreducible curve projecting dominantly onto both coordinates, i.e., the canonical projections $\pi_i:\bP^1\times \bP^1\lra \bP^1$ (for $i=1,2$) restrict to dominant morphisms $(\pi_i)|_C:C\lra \bP^1$. By abuse of notation, we denote the restriction $(\pi_i)_C$ also by $\pi_i$. We define probability measures $\tilde{\mu}_{i,f}$ (for $i=1,2$) on $C$ corresponding to the dynamical system $(\bP^1_{\C}, f)$. For each $i=1, 2$, we  pullback $\mu_f$ by $\pi_i$ to get a measure $\pi_i^*\mu_f$ on $C$ in the following sense
   $$\pi_i^*\mu_f(A):=\mu_f(\pi_i(A))$$
for any Borel set $A\subset C$. Hence we get probability measures on $C$
  $$\tilde{\mu}_{i,f}:=\pi_i^*\mu_f/\deg(\pi_i)\text{ for $i=1,2$.}$$
When there is no confusion on what projection $\pi_i$ we used to construct $\tilde{\mu}_{i,f}$ we drop the index $i$ from our notation and simply use $\tilde{\mu}_f$ to denote the corresponding probability measure on $C$ induced by the dynamical system $(\bP^1, f)$.


\subsection{Symmetries of the Julia set}\label{sym}

Let $H$ be a meromorphic function on some disc $B(a, r)$ of radius $r$ centred at a point $a\in J_f$. We say that $H$ is a {\em symmetry} on $J_f$ if it satisfies
\begin{itemize}
\item $x\in B(a, r)\cap J_f$ if and only if $H(x)\in H(B(a,r))\cap J_f$.
\item When $J_f$ is either a circle, line segment, or entire sphere, there is a constant $\alpha>0$ such that for any Borel set $A$ where $H|_A$ is injective, one has $\mu_f(H(A))=\alpha \cdot \mu_f(A)$. 
\end{itemize}
A family $\mathcal{H}$ of symmetries of $J_f$ on $B(a, r)$ is said to be {\em nontrivial} if $\mathcal{H}$ is normal on $B(a,r)$ and no limit function of $\mathcal{H}$ is a constant. A rational function is \emph{post-critically finite} (sometimes called critically finite), if each of its critical points has finite forward orbit, i.e. all critical points are preperiodic. According to Thurston \cite{Thu, DH}, there is an orbifold structure on $\bP^1$ corresponding to each post-critically finite map. A rational function is post-critically finite with {\em parabolic} orbifold if and only if it is exceptional; or equivalently its Julia set is smooth (circle, line segment or entire sphere) with smooth maximal entropy measure on it; see \cite{DH}. Actually, according to Hamilton \cite{Ham}, a Julia set which is a one-dimensional topological manifold must be either a circle, closed line segment (up to an automorphism of $\P^1$) or of Hausdorff dimension greater than one. By a theorem of  Zdunik \cite{Zdu},  a rational function $f$  is Latt\`es if and only if $J_f$ is $\P^1$ and $\mu_f$ is absolutely continuous with respect to Lebesgue measure on $\P^1$. 
In \cite[Theorem~1]{Levin}, it was shown that there exists an infinite nontrivial family of symmetries on $J_f$ if and only if $f$ is post-critically finite with parabolic orbifold.


  \section{Adelic metrized line bundles and the equidistribution of small points}
\label{section equidistribution setup}

  In this section, we setup the  height functions and state the equidistribution theorem for points of small height, which would be used later in proving the main theorems of this article. The main tool we use here is the arithmetic equidistribution theorem for points with small height on algebraic curves (see \cite{C-L, Yuan}).

\subsection{The height functions}\label{height subsection} Let $K$ be a number field and $\Kbar$ be the algebraic closure of $K$. The number field $K$ is naturally equipped with a set $\OK$ of pairwise inequivalent nontrivial absolute values, together with positive integers $N_v$ for each $v\in \OK$ such that
\begin{itemize}
\item for each $\alpha \in K^*$, we have $|\alpha|_v=1$ for all but finitely many places $v\in \OK$. 
\item every $\alpha \in K^*$ satisfies the {\em product formula}
\begin{equation}\label{product formula}
   \prod_{v\in \OK} |\alpha|_v^{N_v}=1
   \end{equation}
\end{itemize}
For each $v\in \OK$, let $K_v$ be the completion of $K$ at $v$, let $\Kbar_v$ be the algebraic closure of $K_v$ and let $\C_v$ denote the completion of $\Kbar_v$. We fix an embedding of $\Kbar$ into $\C_v$ for each $v\in \OK$; hence we have a fixed extension of $|\cdot |_v$ on $\Kbar$. When $v$ is archimedean, then $\C_v\cong \C$. Let $f\in K(z)$ be a rational function with degree $d\geq 2$. There is a {\em canonical height} $\hhat_f$ on $\P^1(\Kbar)$ given by 
 \begin{equation}\label{C-S height}
\hhat_f(x):=\frac{1}{[K(x):K]} \lim_{n\to \infty}\sum_{\tilde{y}\in \Gal(\Kbar/K)\cdot \tilde{x}}~ \sum_{v\in \OK}N_v\cdot \frac{ \log\|F^n(\tilde{y})\|_v}{d^n}
\end{equation}
 where $F:K^2\to K^2$ and $\tilde{x}$ are homogenous lifts of $f$ and $x\in \P^1(\Kbar)$, and $\|(z_1, z_2)\|_v:=\max\{|z_1|_v, |z_2|_v\}$. By product formula (\ref{product formula}), the height $\hhat_f$ does not depend on the choice of the homogenous lift $F$ and therefore it is well-defined. As proven in   \cite{C-S}, $\hhat_f(x)\geq 0$ with equality if and only if $x$ is preperiodic under the iteration of $f$.

  
\subsection{Adelic metrized line bundle} 
\label{subsection adelic metrized line bundles}

Let $\cL$ be an ample line bundle of an irreducible projective curve $C$ over a number field $K$. As in Subsection~\ref{height subsection}, $K$ is naturally equipped with absolute values $|\cdot|_v$ for $v\in \OK$. A {\em metric} $\|\cdot\|_v$ on $\cL$ is a collection of norms, one for each $x\in X(K_v)$, on the fibres  $\cL(x)$ of the line bundle, with 
   $$\|\alpha s(x)\|_v=|\alpha|_v\|s(x)\|_v$$
for any section $s$ of $\cL$. An {\em adelic metrized line bundle} $\cLbar=\{\cL, \{\|\cdot\|_v\}_{v\in \OK}\}$ over $\cL$ is a collection of metrics on $\cL$, one for each place $v\in \OK$, satisfying certain continuity and coherence conditions; see \cite{Zhang:line, Zhang:metrics}. 

There are various adelic metrized line bundles; the simplest such adelic metrized line bundle is the line bundle $\O_{\P^1}(1)$ equipped with metrics $\|\cdot\|_v$ (for each $v\in \OK$), which evaluated at a section $s:=u_0X_0 + u_1X_1$ of $\O_{\bP^1}(1)$ (where $u_0, u_1$ are scalars and $X_0, X_1$ are the canonical sections of $\O_{\bP^1}(1)$) is given by 
   $$\|s([x_0: x_1])\|_v:=\frac{|u_0x_0+u_1x_1|_v}{\max\{|x_0|_v, |x_1|_v\}}.$$

Furthermore, we can define other metrics on $\O_{\P^1}(1)$ corresponding to a rational function $f$ of degree $d\geq 2$ defined over $K$. We fix  a homogenous lift  $F: K^2\to K^2$ of $f$ with homogenous degree $d$. For $n\geq 1$, write $F^n=(F_{0, n}, F_{1, n})$. For each place $v\in \OK$, we can define a metric on $\O_{\P^1}(1)$ as
   $$\|s([x_0: x_1])\|_{v, F, n}:=\frac{|u_0x_0+u_1x_1|_v}{\max\{|F_{0, n}(x_0, x_1)|_v, |F_{1, n}(x_0, x_1)|_v\}^{1/d^n}},$$
where $s=u_0X_1+u_1X_1$ with $u_0, u_1$ scalars and $X_0, X_1$ canonical sections of $\O_{\P^1}(1)$. Hence $\{\O_{\P^1}(1),\{ \|\cdot\|_{v, F, n}\}_{v\in \OK}\}$ is an adelic metrized line bundle over $\O_{\P^1}(1)$. 

A sequence $\{\cL, \{\|\cdot\|_{v,n}\}_{v\in \OK}\}_{n\geq 1}$ of adelic metrized line bundles over an ample line bundle $\cL$ on a curve $C$ is convergent to $\{\cL, \{\|\cdot\|_v\}_{v\in \OK}\}$, if for all $n$ and all but finitely many $v\in\OK$, $\|\cdot\|_{v,n}=\|\cdot\|_v$, and if $\left\{\log\frac{\|\cdot\|_{v,n}}{\|\cdot\|_v}\right\}_{n\geq 1}$ converges to $0$ uniformly on $C(\C_v)$ for all $v\in \OK$. The limit $\{\cL, \{\|\cdot\|_v\}_{v\in \OK}\}$ is an adelic metrized line bundle. 

A typical example of a convergent sequence of adelic metrized line bundles is $\{\{\O_{\P^1}(1),\{ \|\cdot\|_{v, F, n}\}_{v\in \OK}\}\}_{n\geq 1}$ which converges to the metrized line bundle denoted by $\{\O_{\P^1}(1), \{\|\cdot\|_{v,F}\}_{v\in\OK}\}$  (see \cite{BR} and also see  \cite[Theorem~2.2]{Zhang:metrics} for the more general case of a polarizable endomorphism $f$ of a projective variety). 

Let $f$ be a rational function of degree $d\ge 2$ defined over the number field $K$, let $C$ be a curve equipped with a non-constant morphism $\psi:C\to \P^1$ defined over $K$, and $\cL_\psi:=\psi^*\O_{\P^1}(1)$ be the pullback line bundle by $\psi$. For a fixed homogenous lift $F$ of $f$ with degree $d\geq 2$ and places $v\in \OK$, the metrics $\|\cdot\|_{v, \psi, F, n}$ on $\cL_\psi$ are defined as $\|\cdot \|_{v,\psi, F,n}=\psi^*\|\cdot \|_{v,F,n}$. As the sequence $\{\O_{\P^1}(1),\{ \|\cdot\|_{v, F, n}\}_{v\in \OK}\}_{n\geq 1}$ is convergent, we get a sequence of convergent adelic metrized line bundles $\{\cL_{\psi},\{ \|\cdot\|_{v, \psi, F, n}\}_{v\in \OK}\}_{n\geq 1}$ and use 
  $$\cLbar_{\psi, F}:=\{\cL_{\psi},\{ \|\cdot\|_{v, \psi, F}\}_{v\in \OK}\}$$
to denote its limit, which is an adelic metrized line bundle over $\cL_\psi$ such that 
\begin{equation}
\label{definition metric 2}
 \|\cdot\|_{v, \psi, F}=\psi^*\|\cdot\|_{v,F}.
\end{equation}
For sections $s$ of $\cL_\psi$ of the form $\psi^*(u_0X_0+u_1X_1)$ we have that 
\begin{equation}
\label{definition metric 1}
\|s(t)\|_{v, \psi, F, n}:=\frac{|u_0\psi_0(t)+u_1\psi_1(t)|_v}{\max\{|F_{0, n}(\psi_0(t), \psi_1(t))|_v, |F_{1, n}(\psi_0(t), \psi_1(t))|_v\}^{1/d^n}},
\end{equation}
where $\psi(t)=[\psi_0(t): \psi_1(t)]$, and so,  
$$
\|s(t)\|_{v,\psi,F} = \frac{|u_0\psi_0(t)+u_1\psi_1(t)|_v}{\lim_{n\to\infty} \max\{|F_{0, n}(\psi_0(t), \psi_1(t))|_v, |F_{1, n}(\psi_0(t), \psi_1(t))|_v\}^{1/d^n}}.
$$


\subsection{Equidistribution of small points}
\label{subsection equidistribution small points}

For a semipositive line bundle $\cLbar$ on a (irreducible) curve $C$ defined over a number field $K$, the height for $x\in C(\Kbar)$ is given by 
\begin{equation}\label{points height}
\hhat_{\cLbar}(x)=\frac{1}{|\Gal(\Kbar/K)\cdot x|}\sum_{y\in\Gal(\Kbar/K)\cdot x}~\sum_{v\in \OK}-N_v\log\|s(y)\|_v
\end{equation}
where $|\Gal(\Kbar/K)\cdot x|$ is the number of points in the Galois orbits of $x$, and $s$ is any meromorphic section of $\cL$ with support disjoint from $\Gal(\Kbar/K)\cdot x$. A sequence of points $x_n\in C(\Kbar)$ is {\em small}, if $\lim_{n\to \infty} \hhat_{\cLbar}(x_n)=\hhat_{\cLbar}(C)$; see \cite{Zhang:metrics} for more details on constructing the height for any irreducible subvariety $Y$ of $C$ (which is denoted by $\hhat_\cLbar(Y)$). We use the following equidistribution result due to Chambert-Loir \cite{C-L} (in the case of curves) and to Yuan \cite{Yuan} (in the case of an arbitrary projective variety). 

\begin{theorem}\cite[Theorem 3.1]{Yuan}\label{Yuan equidistribution}
Suppose $C$ is a projective curve over a number field $K$, and $\cLbar$ is a metrized line bundle over $C$ such that $\cL$ is ample and the metric is semipositive. Let $\{x_n\}$ be a non-repeating sequence of points in $C(\Kbar)$ which is small. Then for any $v\in \OK$, the Galois orbits of the sequence $\{x_n\}$ are equidistributed in the analytic space $C^{an}_{\C_v}$ with respect to the probability measure $d\mu_v=c_1(\cLbar)_v/\deg_\cL(C)$. 
\end{theorem}

When $v$ is archimedean, $C_{\C_v}^{an}$ corresponds to $C(\C)$ and the curvature $c_1(\cLbar)_v$ of the metric $\|\cdot\|_v$ is given by $c_1(\cLbar)_v=\frac{\partial \overline{\partial}}{\pi i}\log \|\cdot\|_v$. For non-archimedean place $v$, $C_{\C_v}^{an}$ is the Berkovich space associated to $C(\C_v)$, and Chambert-Loir \cite{C-L} constructed an analog of curvature on $C_{\C_v}^{an}$. The precise meaning of the equidistribution above is that  
   $$\lim_{n\to \infty} \frac{1}{|\Gal(\Kbar/K)\cdot x_n|}\sum_{y\in \Gal(\Kbar/K)\cdot x_n} \delta_{y}=\mu_v,$$
where $\delta_{y}$ is point mass probability measure supported on $y\in C^{an}_{\C_v}$, and the limit is the weak limit for probability measures on the compact space $C^{an}_{\C_v}$. 

For a dynamical system $(\bP^1, f)$, at an archimedean place $v$, it is well known that the curvature of the limit of metrized line bundles 
  $$\{\O_{\P^1}(1),\{ \|\cdot\|_{v, F, n}\}_{v\in \OK}\}_{n\geq 1}$$
  is a $(1,1)$-current given by $d\mu_f$, which is independent of the choice of $F$. We conclude this section by noting that in the case of the metrized line bundle $\cLbar_{\psi, F}$ on a curve $C$ associated to a morphism $\psi:C\lra \bP^1$, then at an archimedean place $v$ we have that 
\begin{equation}
\label{equation c_1}
c_1(\cLbar_{\psi, F})_v=\psi^*{d\mu_f}, 
\end{equation}
where $\mu_f$ is the invariant measure on $\bP^1_{\C_v}$ associated to the rational function $f$ (see Subsection~\ref{subsection measures}). Note that since $v$ is archimedean, then $\C_v=\C$ and so, the equality from  \eqref{equation c_1} follows simply by taking $\frac{\partial \overline{\partial}}{\pi i}\log \|\cdot\|_{\psi, F,v}$ which yields $\psi^*d\mu_f$ by the definition of the metric $\|\cdot\|_{\psi, F, v}$.


\section{Equal measures}
\label{section equal measures}

We work under the assumption from the direct implication in Theorem~\ref{main result}, under the additional assumption that the rational functions $f_i$ and the curve $C$ are all defined over $\Qbar$.  

So, $f_1$ and $f_2$ are rational functions each of degree $d_i>1$ defined over a number field $K$. Also, $C\subset \bP^1\times \bP^1$ is a (irreducible) curve defined over $K$, which projects dominantly onto both coordinates. We assume $C$ contains infinitely many  points of small height, i.e. there exists an infinite sequence of points  $(x_{1,n}, x_{2,n})\in C(\Kbar)$ such that $\lim_{n\to\infty} \hhat_{f_1}(x_{1,n}) + \hhat_{f_2}(x_{2,n})=0$. 

We let $\pi_i:C\lra \bP^1$ for $i=1,2$ be  projection maps of $\bP^1\times \bP^1$ (restricted on $C$) onto both coordinates.  For each $i=1,2$, we let $\cLbar_i:=\cLbar_{\pi_i, F_i}$ be the adelic metrized line bundle on $\cL_i:=\pi_i^*\O_{\bP^1}(1)$ constructed as in Subsection~\ref{subsection adelic metrized line bundles} with respect to the rational function $f_i$. Then for each point $t\in C(\Qbar)$, using the product formula (\ref{product formula}) and equations (\ref{C-S height}, \ref{definition metric 2}, \ref{points height}) we have
\begin{equation}
\label{formula for canonical height}
\hhat_{\cLbar_i}(t)= \hhat_{f_i}(\pi_i(t)).
\end{equation}
Using \cite[Theorem~(1.10)]{Zhang:metrics} along with \eqref{formula for canonical height}, also coupled with the fact that $C$ projects dominantly onto each factor of $\bP^1\times\bP^1$ and the fact that there are infinitely many preperiodic points for $f_i$ (which thus have canonical height $0$), we conclude that 
\begin{equation}
\label{for canonical height 2}
\hhat_{\cLbar_i}(C)=0\text{ for each $i=1,2$.} 
\end{equation}

We let $\mu_i$ be the invariant measure on $\bP^1$ corresponding to the rational function $f_i$, for $i=1,2$; then we let $\tilde{\mu}_i:=\left(\pi_i^*\mu_i\right)/\deg(\pi_i)$. 

\begin{proposition}
\label{equal measures proposition}
With the above notation, $\tilde{\mu}_1 = \tilde{\mu}_2$.
\end{proposition}

\begin{proof}
This is an immediate consequence of Theorem~\ref{Yuan equidistribution} applied to the sequence of points $t_n:=(x_{1,n}, x_{2,n})$ with respect to the adelic metrized line bundles $\cLbar_1$ and $\cLbar_2$. Indeed, using \eqref{formula for canonical height} and \eqref{for canonical height 2} and the assumption on the points $(x_{1,n}, x_{2,n})$ we get that $$\lim_{n\to\infty}\hhat_{\cLbar_i}(t_n)=\hhat_{\cLbar_i}(C),$$ for each $i=1,2$. Therefore we can apply Theorem~\ref{Yuan equidistribution} and conclude that the points $\{t_n\}$ are equidistributed on $C$ with respect to both measures $\tilde{\mu}_1$ and $\tilde{\mu}_2$ (see also \eqref{equation c_1} which yields that $\tilde{\mu}_i$ are indeed the probability measures on $C(\C)$ with respect to which the points $\{t_n\}$ are equidistributed). This concludes the proof of Proposition~\ref{equal measures proposition}.
\end{proof}


  \section{From equal measures to preperiodic curves}
\label{section from identical to preperiodic}

In this section we pick up exactly where Proposition~\ref{equal measures proposition} led us:  we have a curve $C$ projecting dominantly onto each coordinate of $\bP^1\times \bP^1$, and we have two rational functions of degree greater than $1$ with the property that $\tilde{\mu}_f=\tilde{\mu}_g$, where $\tilde{\mu}_f$ and $\tilde{\mu}_g$ are the measures induced on $C$ by the dynamical systems $(\bP^1, f)$ respectively $(\bP^1, g)$ through the canonical projection maps of $C$ onto each of the two  coordinates of $\bP^1\times \bP^1$. The results of this section are valid when $C$, $f$ and $g$ are defined over $\C$ (thus not necessarily over a number field), as long as $\tilde{\mu}_f=\tilde{\mu}_g$. We  prove the following crucial result, which will be used later in the proof of Theorem \ref{main result}. We thank one of the referees for pointing out that we needed a strengthening of our previous version of the following result in order to apply it correctly in the proof of Theorem~\ref{main result}.
  
\begin{theorem}\label{measure to periodicity}
Let  $f$ and $g$ be rational functions defined over $\C$ with degrees greater than one, and assume $f(x)$ is not exceptional. Let $C\subset \mathbb{P}^1\times\mathbb{P}^1$ be an irreducible curve defined over $\C$ which projects dominantly onto both coordinates, and let $\tilde{\mu}_f$ and $\tilde{\mu}_g$ be the induced measures on $C$ by the dynamical systems $(\bP^1, f)$ and $(\bP^1, g)$. Suppose that $\tilde{\mu}_f=\tilde{\mu}_g$ and also suppose that there is a point $(x_0, y_0)\in C$ such that: 
\begin{itemize}
\item $x_0$ is a repelling fixed  point of $f$ and $y_0$ is fixed by $g$; and 
\item there is a non-constant holomorphic germ $h(x)$, with $h(x_0)=y_0$ and $(x, h(x))\in C$ for all $x$ in a small neighbourhood of $x_0$.
\end{itemize}
Then $C$ is fixed by the endomorphism $(x,y)\mapsto (f^n(x),g^m(y))$ of $\bP^1\times \bP^1$ for some positive integers $n$ and $m$, and in particular $\deg (f)^n=\deg(g)^m$. 
\end{theorem}

\begin{proof}
We prove this theorem following an idea of Levin from  \cite{Levin}, where he studied the symmetries of the Julia set of a rational function.  The following proposition is central to our proof.

\begin{proposition}\label{identical germ}
With the same assumptions as in Theorem \ref{measure to periodicity}, there exist integers $n, m\geq 1$, such that $f^n=h^{-1}\circ g^m\circ h$ as germs at $x_0$. 
\end{proposition}

\begin{proof}[Proof of Proposition~\ref{identical germ}.] 
As $f$ is repelling at $x_0$, the multiplier $\lambda_1:=f'(x_0)$ have absolute value greater than one, i.e., $x_0$ is in the support of $\mu_f$ hence $(x_0, y_0)$ is in the support of $\tilde \mu_f=\tilde \mu_g$. Consequently $y_0$ is in the Julia set $J_g$ (the support of $\mu_g$), and then we have $|\lambda_2|\geq 1$ for  the multiplier $\lambda_2:=g'(y_0)$.   As $h(x)$ is not a constant, we define the germ $\tilde{g}$ at $x_0$ as a branch of $\tilde{g}(x):=h^{-1}\circ g\circ h(x)$, which has multiplier $\sqrt[j]{\lambda_2}=\tilde{g}'(x_0)$ where $j$ is the local degree of $h(x)$ at $x_0$. 

First, we show that $|\lambda_2|>1$. Suppose this is not true, then we have $|\lambda_2|=1$. Fix an integer $j_0$ with $\deg(g)^{j_0}>\deg(f)$. As $\tilde \mu_f =\tilde \mu_g$, since $\mu_f(x_0)=0$ and from (\ref{measure growth}), one has  $\mu_f(\tilde g(B))=\deg(g)\cdot \mu_f(B)$ for any measurable set $B$ in a small neighbourhood of $x_0$ with $\tilde{g}$ injective on $B$.  Then for any small neighbourhood $A$ of $x_0$ with both $f$ and $\tilde g$ injective on $A$, one has 
\begin{equation}\label{measure degree growth}
\mu_f(f(A))=\deg(f)\cdot \mu_f(A)\textup{ and } \mu_f(\tilde g(A))=\deg(g)\cdot \mu_f(A)
\end{equation}

Moreover, since $|\tilde g'(x_0)|=1$ and $|f'(x_0)|>1$, we can pick a very small neighbourhood $A$ of $x_0$ with $\tilde g^{j_0}(A)\subset f(A).$ Then one has 
   $$\mu_f(\tilde g^{j_0}(A))=\deg(g)^{j_0}\cdot\mu_f(A)\leq \mu_f(f(A))=\deg(f)\cdot \mu_f(A).$$
Because $\deg(g)^{j_0}>\deg(f)$, we have $\mu_f(A)=0$, which contradicts the fact that $x_0$ is in the support of $\mu_f$. So we have $|\lambda_2|>1$.

As both $|\lambda_1|$ and $|\lambda_2|$ are greater than one, there are invertible holomorphic maps $h_1$ and $h_2$ from a neighbourhood of $x_0$ to $\mathbb{C}$ with $h_1(x_0)=h_2(x_0)=0$, which also satisfy 
   $$h_1\circ f \circ h_1^{-1}(z)=\lambda_1 \cdot z, ~ h_2\circ \tilde{g} \circ h_2^{-1}(z)=\sqrt[j]\lambda_2\cdot z, $$
i.e., in a neighbourhood of $x_0$, both $f$ and $\tilde{g}$ are conjugate to linear functions. 

We pick a sequence of pairs of positive integer $\{(n_\ell, m_\ell)\}_{\ell\geq1}$ with $n_\ell, m_\ell\to \infty$ as $\ell\to \infty$ such that  
    $$\lim_{\ell\to\infty} \frac{\lambda_1^{n_\ell}}{\sqrt[j]\lambda_2^{m_\ell}}=1.$$
Hence the germs  
\begin{equation*}
\begin{split}
f^{n_\ell}\circ \tilde{g}^{-m_\ell}(x) & =[h_1^{-1}\circ (h_1\circ f\circ h_1^{-1})^{n_\ell}\circ h_1]\circ [h_2^{-1} \circ (h_2\circ \tilde{g}\circ h_2^{-1})^{-m_\ell}\circ h_2(x)]\\
 & =h_1^{-1}\circ (h_1\circ f\circ h_1^{-1})^{n_\ell}\circ (h_1\circ h_2^{-1}) \circ (h_2\circ \tilde{g}\circ h_2^{-1})^{-m_\ell}\circ h_2(x)\\
 &=h_1^{-1}\left(\lambda_1^{n_\ell}\cdot h_1\circ h_2^{-1}\left(\frac{h_2(x)}{\sqrt[j]\lambda_2^{m_\ell}}\right)\right)
\end{split}
\end{equation*}
are well defined on $B(x_0, r)$ for some small $r>0$ and all $\ell\geq 1$. We can pick a sufficiently small $r>0$ such that $f^{n_\ell}\circ \tilde{g}^{-m_\ell}|_{B(x_0, r)}$ is injective and $f^{n_\ell}\circ \tilde{g}^{-m_\ell}(B(x_0, r))$ is contained in a small neighbourhood of $x_0$ for all $\ell\geq 1$. Hence $\mathcal{H}:=\{f^{n_\ell}\circ \tilde{g}^{-m_\ell}\}_{\ell\geq 1}$ is a normal family on $B(x_0, r)$. Moreover, as $\tilde{\mu}_f=\tilde{\mu}_g$, then $f^{n_\ell}\circ \tilde{g}^{-m_\ell}$ not only preserves $J_f$ but also $\mu_f$ (up to a scalar multiple) restricted on $B(x_0, r)$, i.e. it is a symmetry on $J_f$.  The multiplier $(f^{n_\ell}\circ \tilde{g}^{-m_\ell})'(x_0)=\lambda_1^{n_\ell}/\sqrt[j]{\lambda_2}^{m_\ell}$ of $f^{n_\ell}\circ \tilde{g}^{-m_\ell}$ at the fixed point $x_0$ tends to $1$ as $\ell$ tends to infinity; thus any limit function of $\mathcal{H}$ would have derivative $1$ at $x_0$ and therefore it will not be a constant function. Consequently, $\mathcal{H}$ contains only finitely many distinct functions on $B(x_0, r)$, otherwise it would contradict to the fact that $f$ is not a post-critically finite map with parabolic orbifold; see  \cite[Theorem 1]{Levin}. We pick integers $\ell_1, \ell_2>1$ with 
   $$n_{\ell_1}>n_{\ell_2}, ~ m_{\ell_1}>m_{\ell_2}, ~ f^{n_{\ell_1}}\circ \tilde{g}^{-m_{\ell_1}}=f^{n_{\ell_2}}\circ \tilde{g}^{-m_{\ell_2}}.$$
Hence on a neighbourhood of $x_0$, one has
   $$f^{n_{\ell_2}}\circ f^{n_{\ell_1}-n_{\ell_2}}=f^{n_{\ell_2}}\circ \tilde{g}^{m_{\ell_1}-m_{\ell_2}},$$
 and then post-compose the germ $f^{-n_{\ell_2}}$ at $x_0$ fixing the point $x_0$ on both sides of  the above equation. This concludes the proof of Proposition~\ref{identical germ}.
\end{proof}

Now consider the analytic equation
   $$h(x)-y=0$$
 on a neighbourhood of $(x_0, y_0)\in \mathbb{P}_\C^1\times \mathbb{P}_\C^1$. The zero set of this equation is an analytic curve passing though the point $(x_0, y_0)$. For $x$ close to $x_0$,  $(x,h(x))$ are points on $C$. Let $n, m$ be integers in Proposition \ref{identical germ}; then $f^n(x)=h^{-1}\circ g^m\circ h(x)$ for all $x$ near $x_0$, i.e.  $ f^n(x)= h^{-1} (g^m(h(x)))$. 

Hence for points $x$ close to $x_0$, the points $(f^n, g^m)(x, h(x))$, which are points on $(f^n, g^m)(C)$  are on the zero set of $h(x)-y=0$. Finally, as both $C$ and $(f^{n}, g^{m})(C)$ share an analytic curve in a neighbourhood of $(x_0, y_0)$, they must be identical. So $C$ is fixed by the endomorphism $(x,y)\mapsto (f^n(x), g^m(y))$ of $\bP^1\times\bP^1$. 

We are left to show that $\deg(f)^n=\deg(g)^m$. There are various ways of proving this last statement, including (as suggested by one of the referees) by  exploiting the equality of degrees for various maps between $C$ and $\bP^1$. We include (as an alternative proof) an analytic argument in the spirit of the rest of our proof of Theorem~\ref{measure to periodicity}. So, we pick a sufficiently small $r>0$, and let $A=B(x_0, r)$. Then both $f^n|_A$ and $\tilde{g}^m|_A$ are injective. From (\ref{measure degree growth}), we have $\mu_f(f^n(A))=\deg(f)^n\cdot \mu_f(A)$, and $\mu_f(\tilde g^m(A))=\deg(g)^m\cdot \mu_f(A)$. Moreover, as $f^n$ and $\tilde g^m$ are identical as germs at $x_0$ and  $\mu_f(A)$ is not zero ($x_0$ is in $J_f$ and the support of $\mu_f$ is $J_f$), we have $\deg(f)^n=\deg(g)^m$. 
\end{proof}

\section{Conclusion of the proof of Theorems~\ref{main result}~and~\ref{general DMM result}}
\label{section proof of main results}

In this section we prove both Theorems~\ref{main result} and \ref{general DMM result}. We first prove the converse implication in Theorems~\ref{main result} and \ref{general DMM result}; i.e., we prove the following result.

\begin{lemma}\label{the converse lemma}
Let $f$ and $g$ be rational functions of same degree $d>1$ defined over $\C$, and let $C\subset \bP^1_\C\times \bP^1_\C$ be a curve which projects dominantly onto both coordinates. If $C$ is preperiodic under the endomorphism $\Phi$ of $\bP^1\times \bP^1$ given by $\Phi(x,y)= (f(x), g(y))$, then $C$ contains infinitely  many preperiodic points.
\end{lemma}

We note that the hypothesis that $C$ is preperiodic under the endomorphism $\Phi$ (and also $C$ is neither horizontal nor vertical) automatically yields the equality of degrees for the two rational functions $f$ and $g$. Also, we remark that an extension of Lemma~\ref{the converse lemma} to polarizable endomorphisms of projective varieties was proven by  Fakhruddin \cite{Fakhruddin}.

\begin{proof}[Proof of Lemma~\ref{the converse lemma}.] As $\Phi$ projects dominantly and $f:\P^1\to \P^1$ has infinitely many preperiodic points on $\P^1$, there are infinitely many points $(x_1, x_2)$ on the curve $C$ with $x_1$ preperiodic under the action of $f$. We want to show that $x_2$ is preperiodic under $g$. As $x_1$ and $C$ is preperiodic under $f$ and respectively  $\Phi$, there are only finitely many distinct points $f^n(x_1)$ and respectively curves $C_n:=\Phi^n(C)$ for $n\geq 1$. Hence there are only finitely many distinct points $(f^n(x_1), g^n(x_2))\in {\pi_1^{-1}}|_{C_n}(f^n(x_1))$ for $n\geq 1$, i.e. $\{g^n(x_2)\}_{n\geq 1}$ is a finite set. Consequently, $x_2$ is preperiodic under $g$. 
\end{proof}

\begin{proof}[Proof of Theorem~\ref{main result}.]
First, we note that the converse implication was already proven in Lemma~\ref{the converse lemma}. 

Secondly, we prove the direct implication in Theorem~\ref{main result} under the assumption that the curve $C$ and the rational functions $f_1, f_2$ are defined over $\Qbar$. Using  \cite[Theorem~1.8]{Mimar}, we get that for each point $(x_1,x_2)\in C(\Qbar)$, we have that $x_1$ is preperiodic for $f_1$ if and only if $x_2$ is preperiodic for $f_2$. In particular, $C$ contains infinitely many preperiodic points.

Since all but finitely many periodic points of a rational map are repelling, and also, at all but finitely many points the projection map on the first coordinate $C\to \bP^1$ is unramified, then we can find a point $(x_0,y_0)\in C$ such that
\begin{itemize}
\item[(a)] $x_0$ is a periodic repelling point for $f_1$; and
\item[(b)] there is a non-constant holomorphic germ $h_0(x)$, with $h_0(x_0)=y_0$ and $(x,h_0(x))\in C$ for all $x$ in a small neighborhood of $x_0$.
\end{itemize}

Note that since $x_0$ is fixed by $f_1$, then $y_0$ must be preperiodic under $f_2$. At the expense of replacing $f_1$ by $F_1:=f_1^{r_1}$ and $f_2$ by $F_2:=f_2^{r_2}$ for some positive integers $r_1,r_2$, and also replacing $C$ by $C_1:=\Phi_{1}(C)$ where $\Phi_{1}$ is the endomorphism of $\bP^1\times \bP^1$ given by $\Phi_{1}(x_1,x_2):=(F_1(x_1),F_2(x_2))$, we can assume that $C_1$ contains a fixed point $(x_{0,1}, x_{0,2})$ of $\Phi_1$  such that $x_{0,1}:=F_1(x_0)$ is a repelling fixed point for $F_1$ (while $x_{0,2}:=F_2(y_0)$ is fixed by $F_2$), and  also ensure that condition~(b) above is satisfied for an analytic germ  in a neighborhood of $(x_{0,1},x_{0,2})\in C_1$, i.e., for the analytic germ $h:=F_2\circ h_0\circ F_1^{-1}$ at $x_{0,1}$, we have $h(x_{0,1})=x_{0,2}$ and $(x,h(x))\in C_1$ for all $x$ in a small neighbourhood of $x_{0,1}$.

Clearly, also $C_1$ contains infinitely many preperiodic points under the action of $\Phi_1$.  
For $i=1,2$, we let $\pi_i:C_1\lra \bP^1$ be the corresponding projection of $\bP^1\times \bP^1$ onto each coordinate, restricted to $C_1$.   We let $\mu_i$ be the invariant measure on $\bP^1(\C)$ corresponding to $F_i$ for $i=1,2$; then we let $\tilde{\mu}_i:=(\pi_i^*\mu_i)/\deg(\pi_i)$ be the corresponding measures on $C_1(\C)$. According to  
Proposition~\ref{equal measures proposition}, we get that $\tilde{\mu}_1=\tilde{\mu}_2$. Therefore the hypotheses of Theorem~\ref{measure to periodicity} are met. Then there exist positive integers $s_1,s_2$ such that $C_1$ is fixed by the endomorphism $(x_1, x_2)\mapsto (F_1^{s_1}(x_1), F_2^{s_2}(x_2))$. In particular, $\deg(F_1^{s_1})=\deg(F_2^{s_2})$. 

Now consider the curve $C$ and the rational functions $G_i:=f_i^{r_is_i}$ for $i=1,2$; then $G_1$ and $G_2$ are of the same degree. Then we run the same argument as above: we replace $G_1$ and $G_2$ by iterates of them, and then replace $C$ by an iterate of it under the action of $(x_1,x_2)\mapsto (G_1(x_1),G_2(x_2))$ such that assumptions in Theorem~\ref{measure to periodicity} are satisfied. Since already $G_1$ and $G_2$ have the same degree, we conclude that $C$ is preperiodic under the action of $(x,y)\mapsto (f_1^{r_1s_1}(x_1), f_2^{r_2s_2}(x_2))$ and $\deg(f_1^{r_1s_1})=\deg(f_2^{r_2s_2})$.

Finally, we deal with the general case of the direct implication from Theorem~\ref{main result}. All we need to prove is that the hypotheses of Theorem~\ref{measure to periodicity} are met (after perhaps replacing each $f_i$ by an iterate and also replacing $C$ by its image under a map of the form $(x_1,x_2)\mapsto (f_1^{r_1}(x_1), f_2^{r_2}(x_2))$). So, all we need to show is that $\tilde{\mu}_1=\tilde{\mu}_2$ where the measures $\tilde{\mu}_i$ on $C(\C)$ are the probability measures on $C(\C)$ induced by the pullbacks of the invariant measures under $f_i$ on $\bP^1_\C$ through the projection maps.  

In this case, let $K$ be a finitely generated subfield of $\C$ such that $f_1$, $f_2$ and $C$ are defined over $K$, and let $\Kbar$ be a fixed algebraic closure of $K$ in $\C$. We know there exists an infinite sequence $S:=\{(x_{1,n}, x_{2,n})\}\subset C(\C)$ such that each $x_{i,n}$ is a preperiodic point for $f_i$ for $i=1,2$ and for $n\ge 1$. Then $f_1$ and $f_2$ are base changes of endomorphisms $f_{1,K}$ and $f_{2,K}$ of $\bP^1_K$; similarly, $S$ is the base change of a subset $S_K\subset C(\Kbar)$. Without loss of generality, we may assume $S_K$ is closed under the action of ${\rm Gal}(\Kbar/K)$ and therefore we can view $S_K$ as a union of closed points in $\bP^1_K\times \bP^1_K$.  We can further extend $f_{i,K}$ to endomorphisms
$$f_{i,U}:  \bP^1_U\lra \bP^1_U$$
over a variety $U$ over $\Q$ of finite type and with function field $K$. 
For each geometric point $t\in U(\Qbar)$, the objects $(f_{1,U}, f_{2,U}, S_U)$ have reductions
$(f_{1,t}, f_{2,t}, S_t)$ such that $S_t$ consists of points with both coordinates preperiodic under the action of $f_{1,U}$, respectively of $f_{2,U}$. We also let $\tilde{\mu}_{i,t}:=\pi_i^*\mu_{f_{i,t}}/\deg(\pi_i)$ (for $i=1,2$) be the probability measures on $C_t$ obtained as pullback through the usual projection map onto each coordinate restricted to $C_t$ of the invariant measures on $\bP^1_\C$ corresponding to each $f_{i,t}$. 
In \cite[Theorem~4.7]{Yuan-Zhang-2} (see also \cite[Lemma~3.2.3]{Yuan-Zhang-1}), Yuan and Zhang prove that the set $S_t$ is still infinite for all the $\Qbar$-points of a dense open subset $U_0\subseteq U$; we note that the results of \cite{Yuan-Zhang-1} were recently published in \cite{Yuan-Zhang-published}, while \cite{Yuan-Zhang-2} has been updated to \cite{Yuan-Zhang-arxiv} using slightly different arguments. Thus, as proven in Proposition~\ref{equal measures proposition}, we conclude that $\tilde{\mu}_{1,t}=\tilde{\mu}_{2,t}$ for each $t\in U_0(\Qbar)$. Since $U_0(\Qbar)$ is dense in $U(\C)$ with respect to the usual archimedean topology, while the measures $\tilde{\mu}_{i,t}$ vary continuously with the parameter $t$ (since from the construction,  the potential functions of these measures vary continuously with the coefficients of $f_{i,t}$), we conclude that $\tilde{\mu}_{1,t}=\tilde{\mu}_{2,t}$ for all points in $U(\bC)$ including the point corresponding to original embedding $K\subset \C$. Thus we have shown that the measures $\tilde{\mu}_1,\tilde{\mu}_2$ induced on $C$ by the action of $f_1$ and $f_2$ on $\bP^1_\C$ are equal. Arguing identically as before, i.e., replacing $f_i$ by iterates and also replacing $C$ by its image under a suitable map $(x_1,x_2)\mapsto (f_1^{r_1}(x_1),f_2^{r_2}(x_2))$, we can guarantee that the hypotheses of Theorem~\ref{measure to periodicity} are met and thus conclude the proof of Theorem~\ref{main result}.
\end{proof}

\begin{proof}[Proof of Theorem~\ref{general DMM result}.]
First we note that the converse implication was proven in Lemma~\ref{the converse lemma}; note also that the tangent subspaces at \emph{each} preperiodic point on the preperiodic curve $C$ is also preperiodic in the corresponding Grassmanian. 

So, from now on, we assume that $C$ contains infinitely many preperiodic points verifying the hypotheses of Theorem~\ref{general DMM result} and we prove next that $C$ is preperiodic. First we note that if the curve does not project dominantly onto both coordinates, then the result is trivial, since in this case it must be that $C$ projects to one coordinate to a preperiodic point under the corresponding map $f_i$.

Let $d:=\deg(f_i)$ (for $i=1,2$). If one of the maps $f_i$ is not exceptional, then the conclusion is provided by Theorem~\ref{main result}; in this case we do not require the hypothesis regarding the action on the tangent subspaces of the preperiodic points.

Assume now that $f_1$ is conjugate to $z^{\pm d}$ or $\pm T_d(z)$. The hypothesis that $C$ contains infinitely many preperiodic points (once again we do not require the extra hypothesis regarding the induced action on the tangent spaces) yields, according to  Proposition~\ref{equal measures proposition}, the equality of the following measures on $C(\C)$: 
\begin{equation}
\label{same measure monomial or chebyshev}
\deg(\pi_2)\cdot \pi_1^*\mu_1=\deg(\pi_1)\cdot \pi_2^*\mu_2, 
\end{equation}
where the $\pi_i$'s are the corresponding projection maps restricted on $C$ to each coordinate of $\bP^1\times \bP^1$,  while $\mu_i$ is the invariant measure on $\bP^1_\C$ corresponding to $f_i$. Then also $f_2$ must be conjugate to $z^{\pm d}$ or $\pm T_d(z)$; see Section \ref{sym}. 

We assume now that both $f_i$ are conjugate to $T_d(z)$; the other cases follow similarly. So, $f_i=\eta_i^{-1}\circ T_d\circ\eta_i$ for some linear transformations $\eta_i$. Then it is sufficient to prove Theorem~\ref{general DMM result} for $f_1$ and $f_2$ replaced by $T_d$ and the curve $C$ replaced by $\Psi_1(C)$, where $\Psi_1:\bP^1\times \bP^1\lra \bP^1\times \bP^1$ is defined by $\Psi_1(x_1, x_2) = (\eta_1(x_1), \eta_2(x_2))$. We let $\nu:\bP^1\lra \bP^1$ be defined by $\nu([a:b])=[a^2+b^2:ab]$, and let $\Psi_2$ be the endomorphism of $\bP^1\times \bP^1$ induced by $\nu\times \nu$. Then an irreducible component $C_1$ of $\Psi_2^{-1}(C)$ contains infinitely many points $(x_1, x_2)$ where each $x_i$ is a preperiodic point for the map $z\mapsto z^d$ on $\bP^1$, i.e. each $x_i$ is a root of unity. Then the classical  Manin-Mumford Conjecture for $\bG_m^2$ proved by Laurent \cite{Laurent}  yields that $C_1$ is a torsion translate of an algebraic subtorus of $\bG_m^2$ and therefore, it is preperiodic under the action of $(x_1,x_2)\mapsto (x_1^d, x_2^d)$ on $\bP^1\times \bP^1$. Thus $C=\Psi_2(C_1)$  is preperiodic under the action of $(f_1, f_2)$. 

Now,  assume $f_1$ is a  Latt\`es map; again using that $\deg(\pi_2)\cdot \pi_1^*\mu_1=\deg(\pi_1)\cdot \pi_2^*\mu_2$ yields that also $f_2$ is a Latt\`es map; see Section \ref{sym}.  So, again at the expense of replacing each $f_i$ by a conjugate of it, and also replacing $C$ by its image under an automorphism of $\bP^1\times \bP^1$, we may and do assume that we have  elliptic curves $E_1$ and $E_2$ defined over $\C$, a coordinatewise projection map $\pi:E_1\times E_2\lra \bP^1\times \bP^1$, and we have a polarizable group endomorphism $\tilde{\Phi}$ of $E_1\times E_2$ (note that $\deg(f_1)=\deg(f_2)>1$) such that $\Phi\circ \pi = \pi\circ \tilde{\Phi}$, where we recall that $\Phi:\bP^1\times \bP^1\lra \bP^1\times \bP^1$ is given by $\Phi(x_1, x_2) = (f_1(x_1), f_2(x_2))$. Then our hypotheses imply the existence of a curve $C_1\subset E_1\times E_2$ (which is a irreducible component of $\pi^{-1}(C)$) such that 
 $C_1$ contains infinitely many smooth preperiodic points $P$ under the action of $\tilde{\Phi}$, and moreover the tangent space of $C_1$ at $P$ is preperiodic under the induced action of $\tilde{\Phi}$ on  ${\rm Gr}_1\left(T_{E_1\times E_2,P}\right)$ (here we use the fact that the ramified locus for $\pi$ is a finite union of divisors of the form $\{a\}\times E_2$ and $E_1\times \{b\}$, while at any unramified point of $\pi$, we get an induced isomorphism $\pi^*$ on the correspnding tangent spaces). Since $\tilde{\Phi}$ is a polarizable endomorphism of $E_1\times E_2$, then the result follows from \cite[Theorem~2.1]{GTZ}.
\end{proof}


\section{Proof of Theorems~\ref{theorem example} and \ref{same poly acting}}
\label{section examples}

In Subsection~\ref{setting up the proof subsection}, we develop the technical results needed in the proof of Theorem~\ref{theorem example} whose proof is finalized in Subsection~\ref{the proof subsection}. We conclude our paper with the proof of Theorem~\ref{same poly acting} in Subsection~\ref{conclude subsection}.


\subsection{A result in polynomial decomposition}
\label{setting up the proof subsection}

Throughout this subsection, fix $d\geq 2$ and $f(x)=x^d+c\in \C[x]$ that is not
	conjugate to $x^d$ or $\pm T_d(x)$. As explained in Section~\ref{sec:introduction}, this amounts to $c\ne 0$ and also $c\ne -2$ if $d=2$. 
	For every positive divisor
	$\delta$ of $d$, let $f_{\delta}(x)=x^{d/\delta}\circ (x^\delta+c)=(x^\delta+c)^{d/\delta}$. Note that $f_d(x)=f(x)$ while
	$f_1(x)=(x+c)^d$ is conjugate to $f(x)$. Fix $n\in \N$. We prove the following:
	\begin{theorem}\label{thm:1}
		Every non-constant solution
		$A,B\in \C[x]$ of the functional equation:
		$$f^n\circ A=A\circ B$$
		 has the form $A=f^m\circ (x^{\delta}+c)\circ L$ and $B=L^{-1}\circ f_{\delta}^n\circ L$
		 for some integer $m\ge 0$, positive divisor $\delta$ of $d$, and linear polynomial $L$.
	\end{theorem}
	
	We start with
	the following lemmas:
	\begin{lemma}\label{lem:gap}
	Let $D\geq 2$ and $2\leq m\leq D$ be integers. Assume
	that $P(x)$ has the form $ax^D+bx^{D-m}+\ldots$
	with $ab\neq 0$. Then for every $n\in\N$, 
	$P^n$ has the form $\alpha x^{D^n}+\beta x^{D^n-m}+\ldots$ with $\alpha\beta\neq 0$.
	\end{lemma}
	\begin{proof}
	This can be proved easily by induction on $n$.
	\end{proof}
		
	\begin{lemma}\label{lem:gcd 1}
		If $A,B\in \C[x]$
		are non-constant such that $f^n\circ A=A\circ B$
		and $\gcd(d,\deg(A))=1$ then $A$ is linear.
	\end{lemma}
	\begin{proof}
		Assume otherwise, let $\alpha:=\deg(A)\geq 2$. By  \cite[Theorem 8]{Inou}, we have
		that $f^n$ is linearly conjugate to 
		  $x^{r}(P(x))^{\alpha}$
		 for some $P(x)\in\C[x]\setminus x\C[x]$ and $r \equiv d^n$ modulo $\alpha$. Since $f^n$ is not linearly conjugate to
		 the power map, we have that $\deg(P)>0$. 
		  Write $F(x)=f^{n-1}(x)$,
		 and let $ux+w$ be a linear polynomial that conjugates $f^n=(F(x))^d+c$ to the above polynomial.
		 We have:
		 $$uF\left(\frac{x-w}{u}\right)^d+uc+w=x^r(P(x))^{\alpha}.$$
	   
	   Assume that $uc+w\neq 0$, then the Mason-Stothers theorem \cite[p.~194]{Lang} (an effective version of the \emph{abc}-conjecture over function fields) implies:
	   \begin{equation}\label{eq:abc1}
	   d^n\leq d^{n-1}+\deg(P).
	   \end{equation}
	   On the other hand, we have $r+\alpha \deg(P)=d^n$. Recall that
	   $\alpha\geq 2$, and $\gcd(\alpha,d)=1$ so that $\alpha\geq 3$ if $d=2$.
	   Hence \eqref{eq:abc1} implies that $\alpha \deg(P)\geq d^n$,
	   contradiction. So we must have $uc+w=0$. Hence $x^r(P(x))^{\alpha}$
	   is a $d$-th power. This is impossible since 
	   $\gcd(r,d)=1$ and $P(0)\neq 0$. 
	 \end{proof}
	 
	 Following \cite[Definition~2.3]{GTZ12}, we say that $g(x,y)$ is a \emph{Siegel factor} of the polynomial $f(x,y)$ if $g$ divides $f$ and the plane curve given by the equation $g(x,y)=0$  is irreducible, has geometric genus $0$, and has at most two points at infinity. The following is an immediate 
	 consequence
	 of \cite[Corollary 2.7]{GTZ12}:
	 \begin{lemma}\label{lem:Siegel factor}
	 Let $m$ and $D$ be coprime integers greater than 1. Let $\tilde{A}$ be a polynomial
	 of degree $D$. If $x^{m}-\tilde{A}(y)$ has a Siegel factor then
	 one of the following holds:
	 \begin{itemize}
	 	\item [(I)] There exist $P(x)\in\C[x]\setminus x\C[x]$, 
	 	$r\in\N$
	 satisfying $r\equiv D$ modulo $m$, and linear polynomial $L(x)$ such that
	 $\tilde{A}(x)=(x^rP(x)^m)\circ L(x)$.
	 	\item [(I')] There exist linear polynomials $L_1(x)$ and $L(x)$ such that
	 	$\tilde{A}(x)=L_1\circ x^D\circ L$. 
	 	\item [(II)] There exist linear polynomials $L_2(x)$, $L_1(x)$, and $L(x)$
	 	such that $x^m=L_1\circ T_m\circ L_2$ and $\tilde{A}=L_1\circ T_D\circ L$.
	 \end{itemize} 
	 Moreover, case (II) can only happen if $m=2$ and $L_1(x)$ 
	 has
	 the form $\ell_1x+2\ell_1$ for some $\ell_1\neq 0$.
	 \end{lemma}
	 \begin{proof}
	 	By \cite[Corollary 2.7]{GTZ12} together with the fact that $\gcd(m,D)=1$,
	 	there exist linear polynomials $L$, $L_1$, $L_2$ and polynomials $F$ and $G$ 
	 	such that $x^m=L_1\circ F \circ L_2$, $\tilde{A}=L_1\circ G\circ L$ and
	 	one of the following holds:
	 	\begin{itemize}
	 		\item [(a)] $F(x)=x^m$ and $G(x)=x^rP(x)^m$ with $P$ satisfying the
	 		conditions in (I).
	 		\item [(a')] $G(x)=x^D$.
	 		\item [(b)] $F(x)=T_m(x)$ and $G(x)=T_D(x)$.
	 	\end{itemize}
	 	Note that cases (a) and (a') come from \cite[(2.7.1)]{GTZ12}, while
	 	case (b) comes from \cite[(2.7.3)]{GTZ12}. It is immediate that (I')
	 	and (II) follow from (a') and (b) respectively. Now assume that (a) holds.
	 	From $x^m=L_1\circ x^m\circ L_2$ and $m\geq 2$, we have that $L_1$ and $L_2$ has zero 
	 	constant
	 	coefficient. Now we ``absorb'' the leading coefficient of $L_1$ into $P(x)$
	 	and (I) follows. 
	 	
	 	For the moreover part, we have that the functional equation $x^m=L_1\circ T_m\circ L_2$
	 	could only hold when $m=2$ by comparing the ramification behavior of $x^m$ and $T_m$.
	 	Write $L_i(x)=\ell_ix+k_i$ for $i=1,2$. From $x^2=\ell_1 T_2(\ell_2 x+k_2)+k_1$,
	 	we have that $k_2=0$ and $k_1=2\ell_1$. Therefore $L_1(x)$ has the form
	 	$\ell_1x+2\ell_1$.
	 \end{proof}

   	The following result of Engstrom \cite[Lemma~3.2]{GTZ12}
   	will be used many times:
   	\begin{lemma}\label{lem:deg divides}
   		Let $A,B,C,D\in\C[x]\setminus \C$ such that 
   		$A\circ B=C\circ D$. 
   		\begin{itemize}
   		\item [(a)] If $\deg(A)\mid \deg(C)$ then
   		there exists a polynomial $P(x)\in \C[x]$ such that
   		$C=A\circ P$ and $B=P\circ D$.
   		\item [(b)] If $\deg(B)\mid \deg(D)$ then
   		there exists a polynomial $Q(x)\in \C[x]$
   		such that $D=Q\circ B$ and $A=C\circ Q$.
   		\end{itemize}
   	\end{lemma} 
	 
	 We now prove Theorem~\ref{thm:1}:
	 \begin{proof}[Proof of Theorem~\ref{thm:1}]
	 	By contradiction, assume $(A,B)$ is a non-constant solution
	 	that does not have the desired form. Assume that $\deg(A)$ is \emph{minimal}
	 	among all counter-examples. 
	 	
	 	If $\gcd(\deg(A),d)=1$ then Lemma~\ref{lem:gcd 1} implies that $A$ is linear.
	 	Then we can write $A(x)=(x+c)\circ L$ and
	 	$B(x)=A^{-1}\circ f^n\circ A=L^{-1}\circ f_1\circ L$, violating
	 	the assumption that $(A,B)$ is a counter-example. If $d\mid \deg(A)$ then by Lemma~\ref{lem:deg divides} 
	 	 there is
	 	$\tilde{A}\in \C[x]$ such that 
	 	$A=f\circ \tilde{A}$ and
	 	$f^{n-1}\circ A=\tilde{A}\circ B$. Hence $f^n\circ \tilde{A}=\tilde{A}\circ B$.
	    Since $(A,B)$ is a counter-example, 
	 	the pair $(\tilde{A},B)$ is also a counter-example. However $\deg(\tilde{A})<\deg(A)$,
	 	contradicting the minimality of $\deg(A)$. Therefore we may assume that
	 	$\delta:=\gcd(\deg(A),d)$ satisfies $1<\delta<d$. 
 
 		Using $f(x)=(x^{\delta}+c)\circ x^{d/\delta}$ and Lemma~\ref{lem:deg divides}, 
 		we have that
 		$A=(x^{\delta}+c)\circ \tilde{A}$
 		and
 		\begin{equation}\label{eq:main}
 			x^{d/\delta}\circ f^{n-1}\circ (x^{\delta}+c)\circ \tilde{A}=\tilde{A}\circ B
 		\end{equation}
 		 for some polynomial $\tilde{A}$. 
 		Since $(A,B)$ is a counter-example, we have that $\tilde{A}$ is not 
 		linear. Put $D=\deg(\tilde{A})\geq 2$, we have that 
 		$\gcd(D,d/\delta)=1$ by the definition of $\delta$. Now Lemma~\ref{lem:Siegel factor}
 		gives that one of the cases (I), (I') or (II) occurs. 
 		
 		\textbf{Case (I):} there exist nonzero polynomial $P(x)$ such that $P(0)\neq 0$,
 		$r\in\N$ such that $r\equiv D$ modulo $d/\delta$, and linear polynomial $L$
 		such that $\tilde{A}=(x^rP(x)^m)\circ L$. Replace $\tilde{A}$ by $\tilde{A}\circ 
 		L^{-1}$
 		and $B$ by $L\circ B\circ L^{-1}$, we may assume that $\tilde{A}=x^rP(x)^m$. Now 
 		\eqref{eq:main} gives that $B(x)^rP(B(x))$ is the $(d/\delta)$-th power
 		of a polynomial. Since $P(0)\neq 0$ and $\gcd(r,d/\delta)$,
 		we have that $B(x)=(\tilde{B}(x))^{d/\delta}=x^{d/\delta}\circ \tilde{B}$
 		for some polynomial $\tilde{B}$.
 		
 		\textbf{Case (I'):} there exist linear polynomials $L_1(x)$ and $L(x)$
 		such that $\tilde{A}=L_1\circ x^D\circ L$. As before, replacing $\tilde{A}$
 		and $B$ by $\tilde{A}\circ L^{-1}$ and $L\circ B\circ L^{-1}$ respectively,
 		we may assume that $\tilde{A}=L_1\circ x^D=\ell_1x^D+k_1$. 
 		Equation \eqref{eq:main}
 		gives:
 		\begin{equation}\label{eq:I'}
 		\ell_1B(x)^D+k_1=Q(x)^{d/\delta}
 		\end{equation}
 		where $Q(x)=f^{n-1}\circ (x^\delta+c)\circ \tilde{A}$. Now assume $k_1\neq 0$. Then the Mason-Stothers Theorem
 	  gives:
 	  $\deg(B)+\deg(Q)-1\geq D\deg(B).$
 		Since $\deg(B)=d^n$ and $\deg(Q)=d^{n-1}\delta D$, the 
 		above inequality implies
 		$\frac{d}{\delta}+D>\frac{d}{\delta}D$
 		which is impossible since $\frac{d}{\delta}$ and $D$ are at least 2. Hence $k_1=0$.
 		By \eqref{eq:I'} and $\gcd(D,d/\delta)=1$ we have that 
 		$B(x)$ is a $(d/\delta)$-th power of a polynomial.
 		
 		Now both Case (I) and Case (I') give that $B(x)=x^{d/\delta}\circ \tilde{B}$
 		for some polynomial $\tilde{B}$. Composing $x^{d/\delta}$
 		to the right of both sides of the functional equation $f^n\circ A=A\circ B$,
 		we have:
 		\begin{equation}\label{eq:compose d/delta}
 		f^n\circ A\circ x^{d/\delta}=A\circ x^{d/\delta} \circ \tilde{B}\circ x^{d/\delta}.
 		\end{equation}
 		Since $d\mid \deg(A)d/\delta$, by Lemma~\ref{lem:deg divides} there is a polynomial $\hat{A}$
 		such that $A\circ x^{d/\delta}=f\circ \hat{A}$
 		and $f^{n-1}\circ A\circ x^{d/\delta}=\hat{A}\circ\tilde{B}\circ x^{d/\delta}$. Hence
 		$f^n\circ \hat{A}=\hat{A}\circ \tilde{B}\circ x^{d/\delta}$.  
 		Since $\deg(\hat{A})=\deg(A)/\delta=D<\deg(A)$ and the minimality of $\deg(A)$ among
 		all counter-examples to Theorem~\ref{thm:1}, there exist $m\in \N_0$,
 		a divisor $\eta$ of $d$, and a linear polynomial $U(x)$ such that
 		$\hat{A}=f^m\circ (x^{\eta}+c)\circ U$
 		and $\tilde{B}\circ x^{d/\delta}=U^{-1}\circ f_{\eta}^n\circ U$. Since $\deg(\hat{A})=D$
 		is coprime to $d/\delta$, we have that $m=0$ and hence $\eta=D\geq 2$. Now
 		$f_{\eta}^n$ is a polynomial in $x^{\eta}$ while
 		$\tilde{B}\circ x^{d/\delta}$ is a polynomial in $x^{d/\delta}$,
 		we have that the constant coefficient of $U(x)$ is zero. By Lemma~\ref{lem:gap}, $\eta$ is
 		the maximum number among all $M$ such that $U^{-1}\circ f_\eta^n\circ U$
 		is a polynomial in $x^M$. On the other hand
 		$U^{-1}\circ f_\eta^n\circ U=B\circ x^{d/\delta}$
 		is a polynomial in $x^{d/\delta}$; we must have that $d/\delta$ divides
 		$\eta=D$. This contradicts $\gcd(D,d/\delta)=1$. It remains to deal with
 		the following:
 		
 		\textbf{Case (II):} assume that $d/\delta =2$ and there exist linear polynomials
 		$L_1$ and $L_2$ such that $\tilde{A}=L_1\circ T_D\circ L$; moreover
 		$L_1$ has the form $\ell_1x+2\ell_1$ for some $\ell_1\neq 0$. As before,
 		replace $\tilde{A}$ by $\tilde{A}\circ L^{-1}$ and $B$ by
 		$L\circ B\circ L^{-1}$, we may assume $\tilde{A}=L_1\circ T_D=\ell_1T_D(x)+2\ell_1$.
 		Since $d/\delta=2$, equation
 		\eqref{eq:main} implies that 
 		 $\tilde{A}\circ B=\ell_1(T_D(B(x))+2)$ is the square of a polynomial. Since $\gcd(D,d/\delta)=\gcd(D,2)=1$,
 		 we have that 
 		$-2$ is a simple root of
 		$T_D(x)+2$. Hence we have that
 		$B(x)+2$ is the square of a polynomial. In other words, there is a 
 		polynomial $\tilde{B}$ such that $B(x)=(x^2-2)\circ \tilde{B}$.
 		We can proceed as in Case (I) and Case (I') as follows.
	  
	  Compose $x^2-2$ to the right of both sides of $f^n\circ A=A\circ B$,
	  we have:
	 	\begin{equation}\label{eq:compose C2}
	 		f^n\circ A\circ (x^2-2)=A\circ (x^2-2)\circ\tilde{B}\circ (x^2-2).
	 	\end{equation}
	 	Recall that $d/2=\delta=\gcd(d,\deg(A))$, hence $d\mid \deg(A\circ(x^2-2))$.
	 	By Lemma~\ref{lem:deg divides}, there is $\hat{A}\in\C[x]$ such that $A\circ (x^2-2)=f\circ\hat{A}$ and
	 	$f^{n-1}\circ A\circ (x^2-2)=\hat{A}\circ \tilde{B}\circ (x^2-2)$, hence $f^n\circ \hat{A}=\hat{A}\circ \tilde{B}\circ (x^2-2)$.
	 	Since $\deg(\hat{A})=D<\deg(A)$ and the minimality of $\deg(A)$ 
	 	among all counter-examples to Theorem~\ref{thm:1}, there exist
	 	$m\in \N_0$, a divisor $\eta$ of $d$, and a linear polynomial 
	 	$U(x)$ such that $\hat{A}=f^m\circ (x^\eta+c)\circ U$
	 	and $\tilde{B}\circ (x^2-2)=U^{-1}\circ f_{\eta}^n\circ U$.
	 	Argue as before, we have that $m=0$, $\eta=D\geq 2$, the constant coefficient
	 	of $U$ is zero, and that $2=d/\delta$ divides $\eta=D$, contradiction. 
	 \end{proof}


	\subsection{Proof of Theorem~\ref{theorem example}}
\label{the proof subsection}

	We have:
	\begin{lemma}\label{lem:is an iterate}
		Let $d\geq 2$ be an integer and let $c\in\C$ be such that
		$x^d+c$ is not linearly conjugate to $x^d$ or $\pm T_d(x)$. Let
		$\delta>1$ be a divisor of $d$ and let
		$g(x)=(x^{\delta}+c)^{d/\delta}$. We have the following:
		\begin{itemize}
		\item [(i)] $g(x)$ is not linearly conjugate to
		$x^d$ or $\pm T_d(x)$.
		\item [(ii)] If $G(x)$ is a polynomial of degree at least 2 having a common iterate
		with $g(x)$ then $G(x)$ is an iterate of $g(x)$.
		\end{itemize}
	\end{lemma}
	\begin{proof}
		If $\delta=d$, then $g(x)=x^d+c$ which is not linearly conjugate to $x^d$ or $\pm T_d(x)$, according to our hypothesis. If $\delta<d$, then the conclusion of part (i) follows from the ramification behavior of
		$x^d$ and $T_d(x)$. It remains to show part (ii).
		
		Let $\tilde{g}$ be a non-linear polynomial of minimal degree such that $\tilde{g}$ commutes with an iterate of $g$. Let $\Gamma$
		be the group of linear polynomials commuting with an iterate of $g$. By \cite[Proposition~2.3]{Nguyen} (also see \cite[Remark~2.4]{Nguyen}), the set of polynomials commuting with an iterate of 
		$g$ is exactly:
		$$\{\tilde{g}^m\circ L:\ m\geq 0,\ L\in \Gamma\}.$$
		It suffices to show that $\Gamma$ is trivial and $\tilde{g}=g$.
	
		From \cite[Proposition~2.3(d)]{Nguyen}, any
		two polynomials that commute with an iterate of 
		$g$ and have the same degree differ by either
		a ``post-composition'' or ``pre-composition''
		by an element of the finite cyclic group 
		$\Gamma$.
		Therefore if $L(x)=ax+b$ is a generator of $\Gamma$, we have 
		$g\circ L=L_1\circ g$ 
		for some $L_1\in\Gamma$, and $L_1=L^D$ for some
		positive integer $D$. By comparing the coefficients from both sides of $g\circ L=L^D\circ g$, we have that $b=0$. Hence $g(ax)=a^Dg(x)$. This implies
		that $a^d=a^D$ and $a^{d-\delta}=a^D$, hence $a$ is a $\delta$-th root of unity. This gives $g\circ L=g$, hence $g^m\circ L=g^m$ for every $m\in\N$. Therefore $\Gamma$ is trivial.
		
		From $g=\tilde{g}^m\circ L$ for some $m$ and $L\in\Gamma$,
		we have $g=\tilde{g}^m$. We assume that $m\geq 2$ and arrive at
		a contradiction. Write $\tilde{\delta}=\deg(\tilde{g})$
		and $\epsilon=\gcd(\delta,\tilde{\delta})$, we have
		$\epsilon\geq 2$ since $\tilde{\delta}^m=d$ is divisible by 
		$\delta$. From $\tilde{g}^m=g=x^{d/\delta}\circ (x^{\delta/\epsilon}+c)\circ x^{\epsilon}$ and Lemma~\ref{lem:deg divides}, we have
		that $\tilde{g}$ is a polynomial in $x^{\epsilon}$. Hence
		there is a maximal $\eta\geq 2$ 
		such that $\tilde{g}$ is a polynomial in $x^\eta$. By 
		Lemma~\ref{lem:gap}, $\eta$ is also the maximal
		number $M$ such that $\tilde{g}^m$ is a polynomial in $x^M$.
		Since $\tilde{g}^m=g$ is a polynomial in $x^{\delta}$,
		we have that $\delta$ divides $\eta$. By Lemma~\ref{lem:deg divides}, there is $P\in \C[x]$ such that
		$\tilde{g}=P\circ (x^{\delta}+c)$ and 
		$x^{d/\delta}=\tilde{g}^{m-1}\circ P.$
		By Lemma~\ref{lem:deg divides} and the fact that 
		$m\geq 2$, we have: 
		$\tilde{g}=x^{\tilde{\delta}}\circ L_1$
		for some linear $L_1\in \C[x]$. Since $\tilde{g}$ is a 
		polynomial in $x^{\delta}$ and $\delta>1$, we must have
		that the constant coefficient of $L_1$ is zero. Hence 
		$\tilde{g}$
		is a monomial, so is $g=\tilde{g}^m$ which yields a contradiction. Therefore $m=1$, $g=\tilde{g}$, and every
		polynomial $G$ commuting with an iterate of $g$ must be
		an iterate of $g$.
	\end{proof}
	
	In view of Theorem~\ref{main result}, Theorem~\ref{theorem example}
	follows from the following:
    
    \begin{proposition}
	 	Let $d_1\geq 2$ and $d_2\geq 2$ be integers. 
	 	Let $F(x)=x^{d_1}+c_1$ and $G(x)=x^{d_2}+c_2$ be 
	 	polynomials that are not linearly conjugate
	 	to $\pm T_d(x)$ or $x^d$. Then there are positive integers
	 	$\ell_1$ and $\ell_2$ such that the self-map 
	 	$F^{\ell_1}\times G^{\ell_2}$
	 	of $(\bP^1)^2$ has a non-trivial periodic curve
	 	if and only if $d_1=d_2$ and $c_1=\zeta c_2$ for a $(d_1-1)$-th 
	 	root of 
	 	unity $\zeta$.
	 \end{proposition}
	 \begin{proof}
	  If $d_1=d_2$ and $c_1=\zeta c_2$  for a $(d_1-1)$-th
	  root of unity then there is a linear polynomial $L$
	  such that $F=L\circ G\circ L^{-1}$. The curve $\{(L(x),x):\ x\in\bP^1\}$ is invariant under $F\times G$. It remains to prove the 
	  ``only if'' part. 
	  
	  For every positive divisor $\delta$ (respectively $\eta$) of $d_1$ (respectively $d_2$), denote 
	  $F_\delta(x)=(x^\delta+c_1)^{d_1/\delta}$ (respectively
	  $G_\eta(x)=(x^\eta+c_2)^{d_2/\eta}$). 
	  
	  By Medvedev-Scanlon \cite[Proposition~2.34]{Medvedev-Scanlon}, there is a positive integer $n$ and non-constant polynomials $A_1$, $A_2$,
	  and $B$ 
	  such that $F^{\ell_1n}\circ A_1=A_1\circ B$ and $G^{\ell_2n}\circ A_2=A_2\circ B$. By Theorem~\ref{thm:1}, 
	  there exist positive divisors $\delta$ of $d_1$ and $\eta$ of 
	  $d_2$ such that
	  $F_{\delta}^{\ell_1n}$ and $G_{\eta}^{\ell_2n}$
	  are linearly conjugate to each other.
	  Since $F_1(x)=(x+c_1)^{d_1}$ and $G_1(x)=(x+c_2)^{d_2}$ are
	  linearly conjugate to $F(x)=F_{d_1}(x)$ and $G(x)=G_{d_2}(x)$ respectively, it
	  suffices to prove the following.
	  
	  \emph{Claim: if there exist a positive integer $n$ and positive divisors $\delta$ of $d_1$ and $\eta$ of $d_2$
	  such that $2\leq \min\{\delta,\eta\}$
	  and $F_{\delta}^{\ell_1n}$ and $G_{\eta}^{\ell_2n}$ are 
	  linearly conjugate to each other then $d_1=d_2$ and $c_1=\zeta c_2$
	  for some $(d_1-1)$-th root of unity $\zeta$}. 
	  
	  Let $U(x)$ be a linear polynomial such that
	  $F_{\delta}^{\ell_1n}=(U^{-1}\circ G_{\eta}\circ U)^{\ell_2n}$.
	  By Lemma~\ref{lem:is an iterate}, we have that 
	  $U^{-1}\circ G_{\eta}\circ U$ is an iterate of
	  $F_{\delta}$ and $U\circ F_{\delta}\circ U^{-1}$ is
	  an iterate of $G_{\eta}$. This gives $d_1=d_2$ (hence 
	  $\ell_1=\ell_2$) and $F_{\delta}=U^{-1}\circ G_{\eta}\circ U$.
	  Write $U(x)=ux+w$ and $d:=d_1=d_2$, we have that $w=0$; otherwise the
	  coefficient of $x^{d-1}$ in $U^{-1}\circ G_{\eta}\circ U$
	  is nonzero while the coefficient of $x^{d-1}$ in $F_{\delta}$ is zero.
	  Hence we have:
	  $(x^{\delta}+c_1)^{d/\delta}=\frac{1}{u}((ux)^{\eta}+c_2)^{d/\eta}.$
	  It follows immediately that $\delta=\eta$ and $u^{d-1}=1$. Compare the
	  coefficients of $x^{d-\delta}$ on both sides and note that $\delta=\eta$, we have:
	  $$\frac{d}{\delta}c_1=\frac{1}{u}\frac{d}{\delta}u^{d-\delta}c_2.$$
	  Since $u$ is a $(d-1)$-th root of unity, we finish the proof of
	  the proposition.
	  \end{proof}


\subsection{Proof of Theorem~\ref{same poly acting}}
\label{conclude subsection}
Let $\tilde{f}$ be a non-linear polynomial of minimal degree commuting with an iterate of $f$ and let $G$ be the
group of linear polynomials commuting with an iterate of $f$.
It is well-known from Ritt's theory that 
$\tilde{f}$ and $f$ have a common iterate and 
$$\{\tilde{f}^m\circ L\colon m\ge 0\text{, } L\in G\}=\{L\circ \tilde{f}^m\colon m\ge 0\text{, } L\in G\}$$
which is exactly the set of polynomials commuting with an iterate of $f$ (see, for example,
\cite[Proposition~2.3]{Nguyen}).

Now assume that $C$ is an irreducible curve in $\bP^1\times\bP^1$ that is neither horizontal nor vertical and contains 
infinitely many $\Phi$-preperiodic points. By Theorem~\ref{main result}, $C$ is preperiodic under $\Phi$, hence
under the self-map $\tilde{\Phi}(x,y):=(\tilde{f}(x),\tilde{f}(y))$ of $\bP^1\times\bP^1$.
Let $N\geq 0$ be such that $\tilde{\Phi}^N(C)$ is 
periodic under $\tilde{\Phi}$. By Medvedev-Scanlon 
\cite[Theorem~6.24]{Medvedev-Scanlon}, $\tilde{\Phi}^N(C)$
is given by the equation
$y=g(x)$ or $x=g(y)$ where $g$ commutes with an iterate
of $f$. Assume we have the equation $x=g(y)$ and write
$g=\ell\circ \tilde{f}^M$, then
$C$ satisfies the equation
$\tilde{f}^{N}(x)=\ell(\tilde{f}^{M+N}(y))$.
Similarly, assume we have the equation $y=g(x)$ for
$\tilde{\Phi}^N(C)$ and write
$g=\ell\circ\tilde{f}^M$, then $C$ satisfies
the equation $\tilde{f}^{M+N}(x)=\ell^{-1}(\tilde{f}^N(y))$. 

Conversely, if $C$ satisfies an equation
of the form $\tilde{f}^n(x)=L(\tilde{f}^m(y))$ where $L$ commutes with an iterate of $f$ then for every
$(\alpha,\beta)\in C$, $\alpha$ is $f$-preperiodic if
and only if $\beta$ is $f$-preperiodic. This gives that $C$
contains infinitely many $\Phi$-preperiodic points.


\end{document}